\documentclass[11pt]{amsart}
\usepackage{amsthm, amssymb, amsmath, amsxtra, latexsym}
\usepackage[all]{xy}
\usepackage{enumerate}
\usepackage{graphicx}

\theoremstyle{plain}
\newtheorem{theorem}{Theorem}[section]
\newtheorem{lemma}[theorem]{Lemma}

\newtheorem{question}[theorem]{Question}

\theoremstyle{remark}

\theoremstyle{definition}

\newtheorem{remark}[theorem]{Remark}

\newtheorem{example}[theorem]{Example}

\newtheorem{conjecture}[theorem]{Conjecture}
\newtheorem{definition}[theorem]{Definition}

\theoremstyle{remark}

\newtheoremstyle{named}{}{}{\itshape}{}{\bfseries}{.}{.5em}{\thmnote{#3}}
\theoremstyle{named}

\begin{document}

\title[Co$\mathcal{CF}$ Groups and Demonstrative Embeddings]{Groups with Context-Free Co-Word Problem and Embeddings into Thompson's Group V} 
\author[Berns-Zieve]{Rose Berns-Zieve} 
 \address{Department of Mathematics \\
 Hamilton College \\
 Clinton, NY 13323}
 \email{rbernszi@hamilton.edu}
\author[Fry]{Dana Fry}
\address{Department of Mathematics and Statistics \\
Mount Holyoke College \\
South Hadley, MA 01075}
\email{fry22d@mtholyoke.edu}
\author[Gillings]{Johnny Gillings}
\address{Department of Mathematics \\
Morehouse College \\
Atlanta, GA 30314}
\email{j.gillingsjr@gmail.com}
\author[Hoganson]{Hannah Hoganson}
\address{Department of Mathematics \\
Miami University \\
Oxford, OH 45056}
\email{hoganshl@miamioh.edu}
\author[Mathews]{Heather Mathews}
\address{Department of Mathematics \\
Miami University \\
Oxford, OH 45056}
\email{mathewhm@miamioh.edu}

\begin{abstract}
Let $G$ be a finitely generated group, and let $\Sigma$ be a finite subset that generates $G$ as a monoid. The \emph{word problem of $G$ with respect to $\Sigma$} consists of all words in the free monoid $\Sigma^{\ast}$ that are equal to the identity in $G$. The \emph{co-word problem of $G$ with respect to $\Sigma$} is the complement in
$\Sigma^{\ast}$ of the word problem. We say that a group $G$ is \emph{co$\mathcal{CF}$} if its co-word problem with respect to some (equivalently, any) finite generating set $\Sigma$ is a context-free language.

We describe a generalized Thompson group $V_{(G, \theta)}$ for each finite group G and homomorphism $\theta$: $G \rightarrow G$. Our group is constructed using the cloning systems introduced by Witzel and Zaremsky. We prove that $V_{(G, \theta)}$ is co$\mathcal{CF}$ for any homomorphism $\theta$ and finite group G by constructing a pushdown automaton and showing that the co-word problem of $V_{(G, \theta)}$ is the cyclic shift of the language accepted by our automaton. 

A version of a conjecture due to Lehnert says that a group has context-free co-word problem exactly if it is a finitely generated subgroup of V. The groups $V_{(G,\theta)}$ where $\theta$ is not the identity homomorphism do not appear to have obvious embeddings into V, and may therefore be considered possible counterexamples to the conjecture.

Demonstrative subgroups of $V$, which were introduced by Bleak and Salazar-Diaz, can be used to construct embeddings of certain wreath products and amalgamated free products into $V$. We extend the class of known finitely generated demonstrative subgroups of V to include all virtually cyclic groups.
\end{abstract}

\keywords{context-free language, pushdown automaton, Thompson's groups}

\subjclass[2010]{20F10, 20E06}

\maketitle

\section{Introduction}
Let $G$ be a group and let $\Sigma \subseteq G$ be a finite set that generates $G$. The \emph{word problem of G with respect to the free monoid $\Sigma^*$} is the set of all words in $\Sigma^*$ that are equivalent to the the identity in G. The \emph{co-word problem of G with respect to $\Sigma^*$} is the complement of the word problem. Both the word problem and the co-word problem of $G$ are languages. The Chomsky Hierarchy \cite{Brookshear} states that the set of regular languages is a subset of context-free languages, the set of context-free languages is a subset of context sensitive languages, and the set of context sensitive languages is a subset of recursive languages. We will focus in particular on context-free languages. A language is \emph{context-free} if it is accepted by a pushdown automaton. If the co-word problem of G is a context-free language, then we say G is co$\mathcal{CF}$. This property does not depend on the choice of monoid generating set. The class of co$\mathcal{CF}$ groups was first studied by Holt, Rees, R\"{o}ver, and Thomas \cite{Holt}. They showed that the class is closed under taking finite direct products, taking restricted standard wreath products with virtually free top groups, and passing to finitely generated subgroups and finite index overgroups. 

One group of particular interest is Thompson's group $V$, which is an infinite but finitely presented simple group. Lehnert and Schweitzer demonstrate that Thompson's group $V$ is co$\mathcal{CF}$. This group is of interest to us because of the conjecture, formulated by Lehnert and revised by Bleak, Matucci, and Neunh\"{o}ffer \cite{Bleak2}, that any group with context-free co-word problem embeds in V, i.e., 

\begin{conjecture}\label{bigone}
Thompson's group $V$ is a universal co$\mathcal{CF}$ group.
\end{conjecture}

In this paper we prove two classes of results, one related to embeddings into V, and the other offering a potential counterexample to Conjecture \ref{bigone}. 

Bleak and Salazar-Diaz \cite{Bleak1} define the class of demonstrative subgroups of $V$ and use this class to produce embeddings of free products and wreath products into $V$. They also show that the class of groups that embed into $V$ is closed under taking finite index overgroups. Their proof of the latter fact appeals to results of Kaloujnine and Krasner \cite{KK}. Here, we use induced actions to give a direct proof. Our argument shows, moreover, that if the original embedding is demonstrative, then so is the embedding of the finite index overgroup. 


A theorem of \cite{Bleak1} says that $\mathbb{Z}$ is a demonstrative subgroup of $V$. The results sketched above prove that 
all virtually cyclic groups are demonstrative, and it appears that these are the only known finitely generated demonstrative subgroups. If $V$ is a universal co$\mathcal{CF}$ group, then it should be possible, by the results of Holt, Rees, R\"{o}ver, and Thomas \cite{Holt}, to find an embedding of $G \wr F_2$ into $V$, where $G$ is co$\mathcal{CF}$ and $F_{2}$ is the free group on two generators. The easiest way to find such an embedding would be to show that $F_2$ has a demonstrative embedding into $V$. We are thus led to ask:

\begin{question} Does there exist a demonstrative embedding of $F_{2}$ into $V$?
\end{question}

Our class of potential counterexamples to Conjecture \ref{bigone} comes from the cloning systems of Witzel and Zaremsky \cite{Witza}. We look at a specific group $V_{(G, \theta)}$ that arises from their family of groups equipped with a cloning system. We define a surjective homomorphism $\Phi$ from $V_{(G, \theta)} \rightarrow V$, which implies that $V_{(G, \theta)}$ acts on the Cantor set. However, by our construction, $V_{(G, \theta)}$ seems to have no obvious faithful actions on the Cantor set when $\theta$ is not the identity homomorphism.

In our main result, we prove that $V_{(G, \theta)}$ is co$\mathcal{CF}$ for all pairs of $\theta$ and finite G. We begin by detailing a construction of a pushdown automaton and we show that the co-word problem is equivalent to the cyclic shift of the language accepted by the automaton, therefore proving that the co-word problem is context-free. 

We briefly outline the paper. Section \ref{background} provides the necessary background for the reader to understand the concepts discussed in the two following sections. In Section \ref{demonstrative}, we give our proofs and examples of all ideas related to demonstrative subgroups. Finally, in Section \ref{maintheorem} we prove the main result of our paper. 

\section{Background} \label{background}

\subsection{Pushdown Automata}

\begin{definition}
Let $\Sigma$ be a finite set, called an \emph{alphabet}. We call elements of the alphabet \emph{symbols}. The \emph{free monoid} on $\Sigma$, denoted $\Sigma^*$, is the set of all finite strings of symbols from $\Sigma$. This includes the empty string, which we denote $\epsilon$. The operation is concatenation. An element of the free monoid is a \emph{word}. A subset of the free monoid is a \emph{language}.
\end{definition}
\begin{example}
Let $\Sigma = \{0,1\}$. The free monoid $\Sigma^*$ contains all finite concatenations of $0$ and $1$ in any order. An example word is $01101$. 
\end{example}
\begin{definition}
Let G be a group. A \emph{finite monoid generating set} is a finite alphabet $\Sigma$ with a surjective monoid homomorphism $\Phi: \Sigma^* \rightarrow G.$ 
The \emph{word problem} of a group G (with respect to $\Sigma$), denoted $WP_\Sigma(G)$, is the kernel of $\Phi$. The complement of the word problem is the \emph{co-word problem}, denoted $CoWP_\Sigma(G)$.
\end{definition}
\begin{definition} 
Let $\Sigma, \Gamma$ be alphabets and let $\#$ be an element of $\Gamma$. A \emph{pushdown automaton} \cite{Brookshear} with stack alphabet $\Gamma$ and input alphabet $\Sigma$ is defined as a directed graph with a finite set of vertices $V$, a finite set of \emph{transitions} (directed edges) $\delta$, an initial state $v_0 \in V$, and a set of \emph{terminal states} $T \subseteq V$. 

\begin{figure}[b]
\center{
\includegraphics[width=250pt]{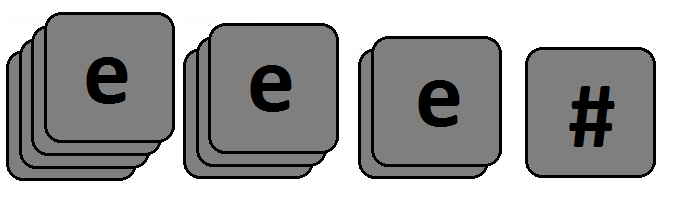}
\caption{The stack of a PDA. As an element is read from the top, the next element appears as the new top if nothing is written to the stack.}
\label{fig:stack}}
\end{figure}
A transition is labeled by an ordered triplet $(w_1,w_2,w_3) \in (\Sigma \cup \{\epsilon \}) \times \Gamma^* \times \Gamma^*$. When \emph{following} a transition, the pushdown automaton (PDA) reads and deletes $w_1$ from its input tape, reads and deletes $w_2$ from its memory stack (shortened to stack for the duration of this paper), and writes $w_3$ on its stack. If $w_1, w_2, \text{ or } w_3 \text{ equals } \epsilon,$ the automaton does not execute the action associated with that coordinate. 
We only consider \emph{generalized} PDA, which (as described above) can add and delete multiple letters on its stack at a time.
A PDA accepts languages either by terminal state, or by empty stack. This must be specified upon creation of the automaton. See Definitions \ref{def:success} and \ref{def:accept}.
\end{definition}
\begin{definition}(\cite{Farley}, Definition 2.6)
Let P be a pushdown automaton. 
We describe a class of directed paths in P, called the \emph{valid paths}, by induction on length. 
The path of length 0 starting at the initial vertex $v_0 \in P$ is valid; 
its stack value is $\# \in \Gamma^*$. 
Let $t_1 \dots t_n (n \geq 0)$ be a valid path in P, where $t_1$ is the transition crossed first. 
Let $t_{n+1}$ be a transition whose initial vertex is the terminal vertex of $t_n$; 
we suppose that the label of $t_{n+1}$ is $(s,w_1,w_2)$. The path $t_1 \dots t_n t_{n+1}$ is also valid, provided that the stack value of $t_1 \dots t_n$ has $w_1$ as a prefix; 
that is, if the stack value of $t_1 \dots t_n$ has the form $w_1 w^{\prime} \in \Gamma^*$ for some $w^{\prime} \in \Gamma^*$. We say that the edge $t_{n+1}$ is a \emph{valid transition}.
The stack value of $t_1 \dots t_n t_{n+1}$ is then $w_2 w^{\prime}$. We let \emph{val(p)} denote the stack value of a valid path $p$.

The \emph{label} of a valid path $t_1 \dots t_n$ is $s_n \dots s_1$, where $s_i$ is the first coordinate of the label for $t_i$ (an element of $\Sigma$, or the empty string). The label of a valid path p will be denoted $\ell(p)$.
\end{definition}
\begin{definition}\label{def:success}
Let $p$ be a valid path of a pushdown automaton P. We say that $p$ is a \emph{successful path} when
\begin{enumerate}
\item $val(p) = \epsilon$ if P accepts by empty stack, or
\item The terminal vertex of $p$ is in T if P accepts by terminal state. 
\end{enumerate}
\end{definition}
\begin{definition}\label{def:accept}
Let P be a PDA. The language \emph{accepted by} P, denoted $\mathcal{L}_P$, is
\item $\mathcal{L}_P = \{w \in \Sigma ^* \mid w = \ell(p) \text{ for some successful path p}\}$.
\end{definition}
\begin{definition}
A subset of the free monoid $\Sigma ^*$ is called a \emph{(non-deterministic) context-free language} if it is $\mathcal{L}_P$ for some pushdown automaton P.
\end{definition}
\noindent Let P be a PDA. We operate P in the following way:
\begin{enumerate}
\item A word $\hat{w}$ is placed on the input tape and read symbol by symbol. By our convention, P reads the input tape from right to left.
\item Next, P follows valid transitions non-deterministically (i.e., by choosing them) until $\hat{w} = \ell(p)$ for some successful path $p$. Throughout this process, the leftmost symbol on the stack is considered to be in the top position.
\item If some successful path $p$ exists, P accepts $\hat{w}$.
\end{enumerate}

\subsection{Thompson's Group V}


\begin{definition}
Let $X = \{0,1\}$ and consider $X^*$. As in \cite{Bleak1}, we define an infinite rooted tree, $\mathcal{T}_2$, as follows:

The set of nodes for $\mathcal{T}_2$ is $X^{\ast}$. For $u, v \in X^*$, there exists an edge from $u$ to $v$ if $ux = v$ for some $x \in X$. 
\end{definition}
\begin{figure}[h]
\center{
\includegraphics[width=250pt]{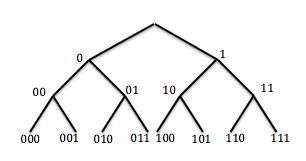}
\caption{The top portion of the infinite binary tree $\mathcal{T}_2$ with some of its nodes labeled.}
\label{fig:tree}}
\end{figure}
\begin{definition}
An \emph{infinite path} $\mathcal{T}_2$ is an infinite string of $0s$ and $1s$. 
\emph{$Ends(\mathcal{T}_2)$} is the collection of all such infinite paths. 

We say that $u \in X^{\ast}$ is a \emph{prefix} of $\omega \in Ends(\mathcal{T}_{2})$
if there is $\hat{\omega} \in Ends(\mathcal{T}_{2})$ such that $\omega = u \hat{\omega}$.
For $u \in X^{\ast}$, we let $u \ast = \{ \omega \in Ends(\mathcal{T}_{2}) \mid u \text{ is a prefix of } \omega \}$.

Define $d:Ends(\mathcal{T}_2) \times Ends(\mathcal{T}_2) \rightarrow \mathbb{R}$
by $d(\zeta_1,\zeta_2) = e^{-l}$, where $\zeta_1, \zeta_2 \in Ends(\mathcal{T}_2)$ and $l$ is the length of the longest prefix shared by $\zeta_1$ and $\zeta_2$. The function $d$ is a metric on $Ends(\mathcal{T}_{2})$.

For $w \in Ends(\mathcal{T}_{2})$, let $B_{r}(w) = \{\zeta \in Ends(\mathcal{T}_{2}) \mid d(\zeta,w) \leq r \}.$ Then $B_{r}(w)$ is the \emph{metric ball around $w$ with radius $r$}. It can be shown that each metric ball in $Ends(\mathcal{T}_{2})$ takes the form $u*$,
for some $u \in X^{\ast}$.

\end{definition}

We note that $Ends(\mathcal{T}_{2})$ is a Cantor set. Thompson's group $V$ acts as self-homeomorphisms on this Cantor set, and each element of $V$ can be represented by a binary tree pair. Furthermore, the leaves of the trees can be represented in binary code where a branch to the left is denoted by ``0" and a branch to the right by ``1."

The group $V$ is generated by the maps $A, B, C $ and $\pi_0$ \cite{CannJ}. We define the generators of V by the prefix changes they represent, which are equivalent to the tree diagrams in Figure \ref{fig:map}. For instance, if $\omega \in Ends(\mathcal{T}_{2})$ has the form $\omega = 0 \hat{\omega}$, for some $\hat{\omega} \in Ends(\mathcal{T}_{2})$, then $A(\omega) = 00 \hat{\omega}$.

\begin{gather*}
\begin{aligned}[c]
0* &\mapsto 00* \\
10* &\mapsto 01* \\
11* &\mapsto 1* \\
\\
&A
\end{aligned}
\begin{aligned}
\indent    \\
\indent    \\
\indent    \\
\indent    \\  
\end{aligned}
\begin{aligned}[c]
0* &\mapsto 0* \\
10* &\mapsto 100* \\
110* &\mapsto 101* \\
111* &\mapsto 11* \\
&B
\end{aligned}
\begin{aligned}
\indent    \\
\indent    \\
\indent    \\
\indent    \\  
\end{aligned}
\begin{aligned}[c]
0* &\mapsto 1* \\
10* &\mapsto 0* \\
11* &\mapsto 10* \\
\\
&C
\end{aligned}
\begin{aligned}
\indent    \\
\indent    \\
\indent    \\
\indent    \\  
\end{aligned}
\begin{aligned}[c]
0* &\mapsto 10* \\
10* &\mapsto 0* \\
11* &\mapsto 11* \\
\\
&\pi_0
\end{aligned}
\end{gather*}
\begin{figure}[ht]
\center{
\includegraphics[width=350pt]{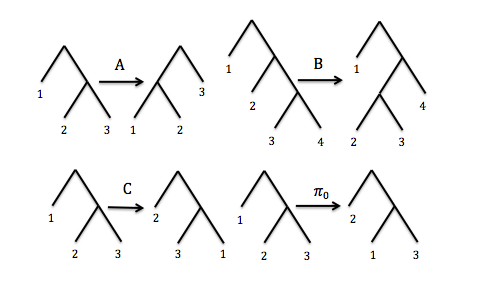}
\caption{Elements $A$, $B$, $C$, and $\pi_0$ of $V$ represented as tree pairs.}
\label{fig:map}}
\end{figure}

\subsection{Generalized Thompson Groups $V_{(G, \theta)}$}

\begin{definition}
\cite{Witza} The \emph{forest monoid}, $\mathcal{F}$, consists of all sequences of ordered, rooted, binary trees $(T_i)_{i\in \mathbb{N}}$, where all but finitely many trees are trivial (i.e., consist only of the root). For two elements $E_1, E_2 \in \mathcal{F}$, their product,  $E_1E_2$, is obtained by attaching the ith leaf of $E_1$ with the ith root of $E_2$.
\end{definition}

Given a finite group, $G$, define
\[H = S_{\infty} \ltimes_\phi (\oplus^{\infty}_{i=1} G)\]
where $\phi: S_{\infty} \to Aut(\oplus^{\infty}_{i=1} G)$ by $\phi(\sigma)(g_1, ... , g_k, ...)=(g_{\sigma^{-1}(1)}, ... g_{\sigma^{-1}(k)}, ...)$ for $\sigma \in S_{\infty}$. 

Multiplication of group elements $(\sigma_1, (g_1, .... , g_k, ...)), (\sigma_2,(g'_1, ... , g'_k, ...)) \in H$ works as follows:

\[(\sigma_1, (g_1, .... , g_k, ...)(\sigma_2,(g'_1, ... , g'_k, ...))=(\sigma_1 \sigma_2, \phi(\sigma_2)(g_1, ... , g_k, ...)(g'_1, ... , g'_k))\]
\[= (\sigma_1 \sigma_2, (g_{\sigma_2^{-1}(1)}g'_1, ... , g_{\sigma_2^{-1}(k)}g'_k, ...))\]

Our group $V_{(G,\theta)}$ arises from the cloning system construction in \cite{Witza}. A \emph{cloning system} consists of a group $H$, a homomorphism $\rho:H \to S_\omega$, and a collection of cloning maps $\{ \kappa_k \mid k \in \mathbb{N} \}$. A cloning system with these three elements that satisfies conditions $CS1, CS2, CS3$ of Proposition 2.7 \cite{Witza} defines a BZS product, $\mathcal{F} \bowtie H$. We do not go into detail on the BZS product; for a full discussion see \cite{Witza}.

Elements of $V_{(G,\theta)}$ are ordered pairs of elements from a subgroup of the BZS product defined by the following cloning system:
\[H = S_{\infty} \ltimes_\phi (\oplus^{\infty}_{i=1} G) \text{,}\]
\[\rho(\sigma, (g_1, ... , g_k))=\sigma \text{,}\]
where $\kappa_k$ acts on the right by:
\[(\sigma, (g_1, ... , g_k,g_{k+1}, ... ))\kappa_k = (\sigma \zeta,(g_1, ... , g_k, \theta(g_k),g_{k+1}, ...))\]
where $\zeta$ is the cloning map for the symmetric group defined in Example 2.9. of \cite{Witza}, and $\theta$ is an arbitrary homomorphism from $G\to G$.

We can think of elements of $V_{(G,\theta)}$ as equivalence classes of tree pairs much as in  Thompson's group $V$. The difference is that tree pairs in $V_{(G,\theta)}$ have group elements from $G$ attached to their leaves. A tree pair $(a,b)$ can be modified within its equivalence class by adding canceling carets or canceling group elements. 

Canceling carets are added to corresponding leaves of the domain and range trees just as they would be in Thompson's group V, unless there is a group element $g$ on the leaf, in which case we put a $g$ on the left branch of the new caret and $\theta(g)$ on the right branch. Canceling group elements are added to corresponding leaves of the domain and range trees and are combined, using the group operation, with any group element already on the leaves. We can multiply two elements $(a,b), (c,d) \in V_{(G,\theta)}$, by choosing equivalent tree pairs $(a',b')$ and $(c',d')$ where $b'=c'$. Then $(a,b)(c,d)=(a',b')(c',d')=(a',d')$. (Note that, in these tree pairs, the second coordinate corresponds to the domain, and the first to the range.)

\begin{figure}[h]
\center{
\includegraphics[width=430pt]{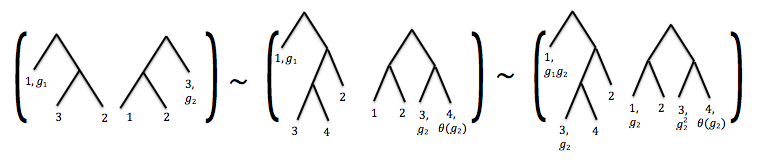}
\caption{Three equivalent tree pairs, the second is obtained by adding a canceling caret and the third by canceling group elements.}
\label{fig:cancelcare}}
\end{figure}

\begin{figure}[h]
\center{
\includegraphics[width=320pt]{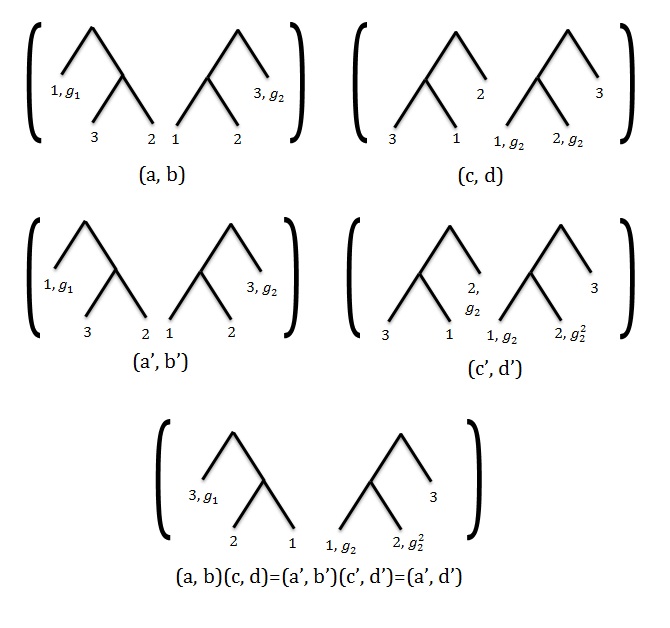}
\caption{Multiplication of two group elements (a,b) and (c,d).}
\label{fig:grultiple}}
\end{figure}

Note that any element $(a,b) \in V_{(G,\theta)}$ can be expressed with no group elements in the domain tree by adding canceling group elements.

If $\{g_1, ... , g_n\}$ is a generating set for $G$, then $\{A,B,C,\pi_0\} \cup \{g_{ja}, g_{jb}, g_{jc}, g_{jd}, g_{je} \mid 1 \le j \le n\}$ is a generating set for $V_{G,\theta}$, where $ g_{ja}, g_{jb}, g_{jc}, g_{jd}, g_{je}$ are defined as in Figure \ref{fig:??}.

\begin{figure}[h]
\center{
\includegraphics[width=350pt]{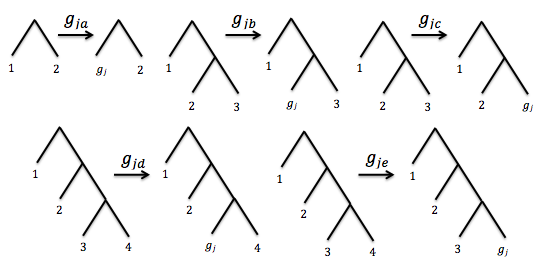}
\caption{Additional generators of $V_{(G, \theta)}$.}
\label{fig:??}}
\end{figure}

\begin{remark}
When $\theta = id_G$, the identity homomorphism,  $V_{(G,\theta)}$ embeds in $V$. Consider $V_{(G,\theta)}$ where $\theta$ is the identity homomorphism, $G=\{g_1, ... , g_n\}$, and $(a,b) \in V_{(G,\theta)}$. We assume further that the domain tree $b$ has no group elements on its leaves. Choose a partition $W=\{w_1*, ... , w_n*\}$ where $w_i$ will be the prefix associated with group element $g_i$. Consider the jth leaf of the tree pair, $b_j* \to a_jg_j*$. If $g_ig_j = g_k$, then $w_ib_j*$ is mapped to $w_kb_j*$. This assignment encodes each $(a,b) \in V_{(G, \theta)}$ as an element of $V$ in an injective fashion. However, when $\theta \not= id_g$, the same method of embedding fails. 


\end{remark}

\section{Demonstrative Groups} \label{demonstrative}

\begin{definition} \cite{Bleak1}
Suppose a group $G$ acts by homeomorphisms on a topological space $X$. For a group $H \le G$, the action of $H$ in $G$ is \emph{demonstrative} if and only if there exists an open set $U \subset X$ such that for all $h_1, h_2 \in G,\text{ } h_1U \cap h_2U \not= \emptyset \text{ if and only if } h_1=h_2$. The set $U$ is called a \emph{demonstration set}. 

For this discussion, let $G = V$ and $X = Ends(\mathcal{T}_2)$. In this case, if $U$ (as above) is a metric ball (i.e. $U = w*$ for some $w* \subseteq Ends(\mathcal{T}_2)$), then we refer to $U$ as a \emph{demonstration node}.
\end{definition}

\begin{definition}
Let $H \le G$ and let $H$ act on a topological space $S$. We define $G \times_H S$ to be $\{(g,s) : g\in G, s \in S\}/\sim$, where $(gh,s) \sim (g, h\cdot s)$. 

The \emph{induced action} of $G$ on $S$ is $*: G \times (G \times_H S) \to G\times_H S$ defined by $g_1*(g_2,s)=(g_1g_2,s)$.
\end{definition}

\begin{theorem}\label{dem1}
If $H \le G \text{ where }[G:H] =m, \text{ for some } m\in \mathbb{N}$, then $G$ embeds in $V$. Moreover, if $H$ embeds as a demonstrative subgroup in $V$, then $G$ embeds as a demonstrative subgroup of $V$.
\end{theorem}
\begin{proof}Assume $H \le G$. Choose a left transversal $T = \{t_1, t_2, ... , t_m\}$ for $H$ in $G$ with $t_1=1$. 
We can induce an action of $G$ on $G \times_H Ends(\mathcal{T}_2)$ by:
\[g \cdot (t_i,x) = (gt_i,x), \text{ for } x \in Ends(\mathcal{T}_2)\]
We know we can write $gt_i$ as $t_jh$ for unique $t_j \in T$ and $h \in H$. So,
\[(gt_i,x)=(t_jh,x) = (t_j, h \cdot x)\]
Now, we can embed $G \times_H Ends(\mathcal{T}_2) \hookrightarrow Ends(\mathcal{T}_2)$ by choosing a set $W = \{ w_{1}, \ldots, w_{m} \}$ such that  $\{w_1\ast , w_2\ast , \dots, w_m\ast\}$ is a partition of $Ends(\mathcal{T}_2)$, and defining an injective function $\phi : T \to W$ by $\phi (t_i)=w_i$. Now define $\Phi: G \times_H Ends(\mathcal{T}_2) \to Ends(\mathcal{T}_2)$ by $\Phi((t_i,x)) = \phi(t_i)x=w_ix$.

It can be easily checked that 
\[g \cdot (w_ix)=g \cdot \Phi(t_i,x)=\Phi(gt_i,x)=\Phi(t_j,h \cdot x)=w_jh(x)\] 
is a group action of $G$ on $Ends(\mathcal{T}_2)$. Additionally, because elements of $H$ act as elements of $V$ and $[G:H] < \infty$, so do elements of $G$. Therefore, $G$ embeds in $V$. 

Now, assume $H$ has a demonstrative embedding in $V$ with demonstration node $a_1a_2 ... a_{n}* \text{ for } a_i \in \{0,1\}$.

We will show that $w_ia_1 \dots a_n*$ is a demonstration node for $G$. We compute the action of each of $g,g' \in G$ on $w_ia_1 \dots a_n*$:
\begin{align*}
g \cdot w_ia_1 \dots a_n* &= w_jh(a_1 \dots a_n*) \\
g' \cdot w_ia_1 \dots a_n* &= w_kh'(a_1 \dots a_n*)
\end{align*}
Since $w_1*, \dots , w_m*$ partition $Ends(\mathcal{T}_2)$, if $w_{j}\ast \cap w_{k}\ast \not=\emptyset$ then we have $j=k$. Since $H$ is a demonstrative subgroup of $G$ with demonstration node $a_1a_2 \dots a_n *$, $h(a_1 \dots a_n \ast) \cap h'(a_1 \dots a_n*) \not= \emptyset$ if and only if $h=h'$. Thus, $w_jh(a_1 \dots a_n*) \cap w_kh'(a_1 \dots a_n*) \not= \emptyset$ if and only if $j=k$ and $h=h'$, in other words, if and only if $g=g'$. Thus, $G$ is demonstrative in $V$ with demonstration node $w_ia_1 \dots a_n* \text { for any } i \in \{1, ... , m\}$. 
\end{proof}

\section{Main Result} \label{maintheorem}

\begin{lemma} \label{testpartition}
Let $\Sigma = \{A,B,C, \overline{A}, \overline{B}, \pi_0,g_{ij}\}$, where  $i \in \{ 1, \ldots, n \}$, and $j \in \{ a, b, c, d, e \}$. Define $\Phi: \Sigma^* \to V$ by the homomorphism 
\begin{align*}
\Phi: A &\mapsto A \\
\bar{A} &\mapsto A^{-1} \\
B &\mapsto B \\
\bar{B} &\mapsto B^{-1} \\
C &\mapsto C \\
\pi_0 &\mapsto \pi_0 \\
g_{ij} &\mapsto 1_V
\end{align*} 
If $w=b_1 \cdots b_m \in \Sigma^*$ satisfies $\Phi(w) \not= 1$, then there is a cyclic permutation $b_j \cdots b_mb_1 \cdots b_{j-1}$ and some $B \in \{a_1a_2a_3*:a_i \in {0,1}\}$ that satisfy $b_j \cdots b_mb_1 \cdots b_{j-1}(B)\cap B \not= \emptyset$. (Here the action of a word $w \in \Sigma^{\ast}$ on the ball $B$ is determined by
the rule $w(B) = \Phi(w)(B)$.)
\end{lemma}

\begin{proof}
Consider $w=b_1 \cdots b_m \in \Sigma^*$. If $\Phi(w) \not= 1$, then $\Phi(w) \in CoWP(V)$. We know $\{a_1a_2a_3*:a_i \in {0,1}\}$ is a test partition for $V$ \cite{Farley}. So, there is some cyclic permutation $\Phi(w)'$ of $\Phi(w)$ and some $B \in \{a_1a_2a_3*:a_i \in \{0,1\}\}$ such that $\Phi(w)'(B) \cap B \not= \emptyset$. Since $\Phi$ takes all the generators $g_{ij}$ to 1, $\Phi$ will preserve the shape of any tree pair. Thus, if $\Phi(w)'$ is such that $\Phi(w)'(B) \cap B \not= \emptyset$, then $w'(B) \cap B \not= \emptyset$ where $w'$ is some cyclic shift $b_j \cdots b_mb_1 \cdots b_{j-1}$ of $w$. 
\end{proof}

\begin{definition}
Let $\mathcal{L}$ be a language. The \emph{cyclic shift} of $\mathcal{L}$, denoted $\mathcal{L}^{\circ}$, is
\[\mathcal{L}^{\circ}=\{w_2w_1 \in \Sigma^* \mid w_1w_2 \in \mathcal{L}, w_1, w_2 \in \Sigma^*\}. \]

A \emph{cyclic permutation} $w^{\prime}$ of a word $w = w_1w_2$ is $w^{\prime} = w_2w_1$. 
Note that the class of context-free languages is closed under cyclic shifts \cite{Maslov}.
\end{definition}

\begin{theorem} \label{main}
The group $V_{(G,\theta)}$ is co$\mathcal{C} \mathcal{F}$.
\end{theorem}

\begin{proof} 
 
We design an automaton $P$ to accept by empty stack, with stack alphabet $\Gamma = \{0, 1, g \mid g \in G \backslash \{1_G\} \}$ and input alphabet $\Sigma = \{A, B, C, \overline{A}, \overline{B}, \pi_0, g_{ij}\}$, where $i \in \{ 1, \ldots, n \}$, and $j \in \{ a, b, c, d, e \}$.

We define 
\[ \mathcal{L}_{B_i} = \{w \in \Sigma^* \mid w(B_i) \cap B_j \not= \emptyset \text{ for some }j \not= i\}. \]
We let $\mathcal{L}_{G}$ be the set of words $w$ in $\Sigma^{\ast}$ such that there is a tree pair representative for $w \in V_{(G, \theta)}$ with no group elements written on the leaves of the domain tree, and at least one non-trivial $g \in G$ written on a leaf of the range tree.

We design $P$ such that $\mathcal{L}_{P}= (\mathop{\bigcup}_{i=1}^8 \mathcal{L}_{B_i}) \cup \mathcal{L}_G $. Figure \ref{fig:bigautomaton} outlines a portion of the automaton. Note that unlabeled arrows represent $(\epsilon, \epsilon, \epsilon)$ transitions. From the initial loading phase, there are in fact eight different arrows $(\epsilon, \epsilon, B_{k})$, one for each of the $B_{k} \in \{  000, 001, \ldots, 111 \}$. These lead to eight separate reading and accept phases, each as pictured in the Figure. These reading and accept phases are identical, with one exception: the labels on the arrows leading to the test partition accept state vary. For instance, in the accept phase corresponding to $000$, the arrow labelled $(\epsilon, B_{l}, \epsilon)$ corresponds to seven different arrows, one
for each $B_{l} = a_{1}a_{2}a_{3}$, where $a_{i} \in \{ 0, 1 \}$ and not all of $a_{1}$, $a_{2}$, $a_{3}$ are $0$.

\begin{figure}[h]
\center{
\includegraphics[width=400pt]{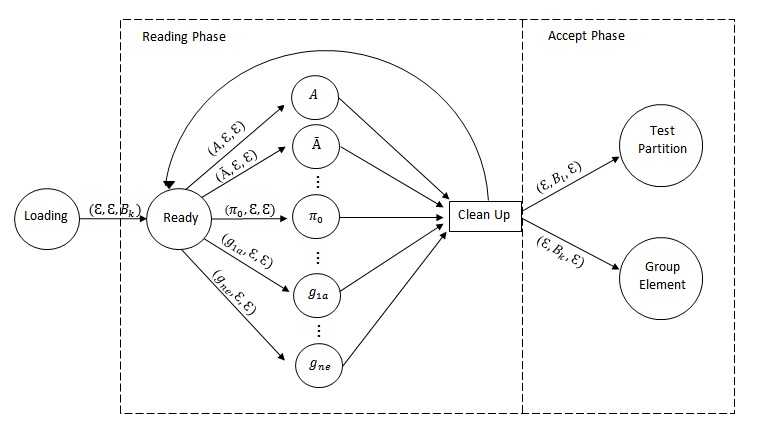}
\caption{Sample reading and accept phases of the automaton for $V_{(G,\theta)}$.}
\label{fig:bigautomaton}}
\end{figure}

To start off, $P$ enters a non-deterministic loading phase. This consists of a single state $S$ with transitions labelled $(\epsilon, \epsilon, 0)$ and $(\epsilon, \epsilon, 1)$, both leading back to $S$. Here a finite string of 1's and 0's is entered non-deterministically onto the  stack. 

We leave the loading phase by taking a transition $(\epsilon, \epsilon, B_k)$ where
\[B_k \in \{000, 001, 010, 100, 011, 110, 101, 111\},\] 
i.e. $B_k$ is one of the 8 metric balls in the test partition. 

Next, $P$ enters the reading phase which has a single state for each generator, where the first element on the input tape is read and that generator is applied to the appropriate prefix at the top of the stack. 
For example, the reading phase for A would read and delete a 0, and then add 00; the reading phase for $g_{2b}$ would read and delete 10, and then add $10g_2$.

After the reading phase, $P$ enters the clean-up state, which consists of the pushing and combining of stack elements. This phase allows $P$ to ``clean-up" the stack so that there are at least three 0's and 1's for the next element on the input tape to successfully act on the stack. First, $P$ ``pushes" elements within the first three spots to the fourth spot on the stack. For example, as shown in Figure \ref{cleanup}, one set of transitions will read and delete $0g0$ or $0g1$ from the stack and then add $00g$ or $01$($\theta(g)$), respectively, to the stack for all $g \in G \backslash \{ 1_G \}$. Similar transitions can be followed if the group element is preceded by the prefix 0, 1, 01, 10, 00, or 11. 

\begin{figure}[ht]
\center{
\includegraphics[scale=.7]{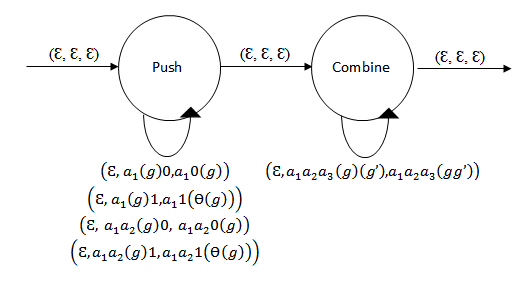}
\caption{The clean-up phase. Let $a_1, a_2, a_3 \in \{0, 1\}$ and $g,g^{\prime} \in G \backslash \{I_G\}$. }}
\label{cleanup}
\end{figure}

Next, $P$ enters the combining state where group elements are rewritten as a single element of the group (in accordance with the group operation). For example, one collection of edges is able to read and delete $010(g)(g')$ and add $010(gg')$ for all $g, g' \in G \backslash \{1_G\}$. Note that if $gg'=1_G$, then the path reads and deletes $010(g)(g')$, and writes 010. Similar paths exist when combining any two group elements preceded by any three-digit prefix.

After exiting the clean-up phase, $P$ reads the next element on the input tape and repeats the process of the reading phase. When the input tape is empty and $P$ has gone through the reading and clean-up phases, $P$ then moves onto one of two accept states.

If the word from the input tape took one metric ball $B_k$ in the test partition to some other metric ball $B_l$, then the three-letter prefix describing $B_{l}$ is now showing on the stack, so we can follow a path labelled $(\epsilon, B_{l}, \epsilon)$ to the \emph{test partition accept state}. (Here we recall that the single arrow labelled $(\epsilon, B_{l}, \epsilon)$ in
Figure \ref{fig:bigautomaton} is actually seven different arrows, one for each $l$ such that $B_{l} \neq B_{k}$.) At this point we ``unload" all of the 0's and 1's and group elements off the stack with paths $(\epsilon,x,\epsilon)$ for $x \in \{0, 1, g \mid g \in G\backslash \{1_G\}\}$. Once every stack element has been deleted, $P$ takes the path $(\epsilon,\#,\epsilon)$ which deletes the start symbol and accepts the word. So the language accepted by the eight test partition accept states is $(\mathop{\bigcup}_{i=1}^8 \mathcal{L}_{B_i})$.

If the word from the input tape does not displace metric ball $B_k$, then we enter the \emph{group element accept state}. Here, we delete every 0 and 1 on the stack until $P$ arrives at a group element. The group element is then``pushed" further down the stack, and the 0 or 1 it pushes past is deleted.  For example, one path is $(\epsilon, g0, g)$ for $g \in G \backslash \{1_G\}$. If the group element is followed on the stack by a second group element, then they are ``combined" in a manner mimicking the previously described combining portion of the clean-up phase. This is repeated until there are no 0's or 1's left on the stack. At this point, if there is still a group element remaining on the stack followed by the start symbol, then they are both deleted and thus the word is accepted. However, if there is no group element on the stack then the start symbol cannot be deleted so the word is not accepted. Assuming that the address of an appropriate metric ball was written on the stack in the loading phase, there will be a group element remaining, and so the language accepted by the eight group element accept states is $\mathcal{L}_G$. 

Thus, $\mathcal{L}_P = (\mathop{\bigcup}_{i=1}^8 \mathcal{L}_{B_i}) \cup \mathcal{L}_G $.

We claim that $\mathrm{CoWP}(V_{(G, \theta)})= (\mathcal{L}_P)^{\circ}$.

Let $w \in \mathcal{L}_P$, so that $w \in  \mathcal{L}_G$ or $w \in \mathcal{L}_{B_{i}}$, for some $i$. If $w \in \mathcal{L}_G$, then it follows directly that $w \in $ CoWP($V_{(G, \theta)}$). If $w \in \mathcal{L}_{B_{i}}$ for some $1\leq i\leq8$, then $w(B_i) \cap B_j$ for $j\not= i$, so $w \in $ CoWP($V_{(G, \theta)})$. Therefore, $\mathcal{L}_P \subseteq$ CoWP($V_{(G, \theta)}$). The CoWP of a group is closed under cyclic shift, and thus $(\mathcal{L}_P)^{\circ} \subseteq$ CoWP($V_{(G, \theta)}$).


Let $w \in $CoWP($V_{(G, \theta)}$). We will use the surjective homomorphism $\Phi$ from Lemma \ref{testpartition}.
If $w \notin Ker(\Phi)$, then $\Phi(w) \not=1_V$. By Lemma \ref{testpartition}, there is some cyclic permutation $w'$ and some $B_i$ such that $w' (B_i) \cap B_i \not= \emptyset$, i.e. $w^{\prime} \in \mathcal{L}_{B_i}$. Therefore, $w \in (\mathcal{L}_{B_i})^{\circ}\subseteq (\mathcal{L}_P)^{\circ}$.
If $w \in Ker(\Phi) \backslash \{1_{V_{(G,\theta)}} \}$, then $\Phi(w)={1_V}$, which implies that $w$ (as a reduced tree pair) does not change any prefixes; it simply adds group elements. Therefore, $w\in{\mathcal{L}_G}\subseteq (\mathcal{L}_P)^{\circ}$.

We now have that CoWP($V_{(G, \theta)}$)= $(\mathcal{L}_P)^{\circ}$. A language is context-free if and only if its cyclic shift is also context-free. Since $\mathcal{L}_P$ is context-free, CoWP($V_{(G, \theta)}$) is context-free, and $V_{(G, \theta)}$ is co$\mathcal{CF}$. 

\end{proof}

\begin{remark}
We attempted a similar method of proof with the group generated by Thompson's group V and the Grigorchuk group G, $R=\langle V,G \rangle$. Like $V_{(G,\theta)}$, elements of $R$ can be thought of as Thompson's group $V$ elements with Grigorchuk group elements, $g$ attached to the leaves. However, the Grigorchuk group elements continue to act on the tree whereas the group elements in $V_{(G,\theta)}$ just sit at the end of the branches. This became a problem because it is impossible to complete the calculation of the action of $g$ on any finite test string loaded onto an automaton. We also ran into problems because the test partitions argument used in Theorem \ref{main}. and for Finite Similarity Structure Groups \cite{Hughes1} fails.
\end{remark}

\subsection*{Acknowledgments}
We would like to first thank our research advisor, Dr. Dan Farley, for aiding us in our research endeavors and providing us with the knowledge needed for us to be successful. 
We thank SUMSRI for giving us the opportunity to participate in undergraduate research in mathematics, and we thank the faculty and staff for making this program possible. In addition, we would like to express our gratitude to Miami University for providing funding, housing, and a place to do research. Finally, we thank the National Science Foundation for funding SUMSRI and making our research possible.

\bibliographystyle{plain}
\bibliography{AlgebraPaper}

\begin{thebibliography}{10}

\bibitem{Bleak2}
Collin Bleak, Francesco Matucci, and Max Neunh\"{o}ffer.
\newblock Embeddings into thompson's group $v$ and $co\mathcal{CF}$ groups.
\newblock {\em arXiv:1312.1855, 15 pages}, 2013.

\bibitem{Bleak1}
Collin Bleak and Olga Salazar-D{\'{\i}}az.
\newblock Free products in r. thompson's group {$V$}.
\newblock {\em Trans. Amer. Math. Soc.}, 365(11):5967--5997, 2013.

\bibitem{Brookshear}
J.~Glenn Brookshear.
\newblock {\em Formal Languages, Automata, and Complexity}.
\newblock Theory of Computation. The Benjamin/Cummings Publishing Company,
  1989.

\bibitem{CannJ}
J.~W. Cannon, W.~J. Floyd, and W.~R. Parry.
\newblock Introductory notes on {R}ichard {T}hompson's groups.
\newblock {\em Enseign. Math. (2)}, 42(3-4):215--256, 1996.

\bibitem{Farley}
Daniel Farley.
\newblock Local similarity groups with context-free co-word problem.
\newblock {\em arXiv:1406.4590, 17 pages}, 2014.

\bibitem{Holt}
Derek~F. Holt, Sarah Rees, Claas~E. R{\"o}ver, and Richard~M. Thomas.
\newblock Groups with context-free co-word problem.
\newblock {\em J. London Math. Soc. (2)}, 71(3):643--657, 2005.

\bibitem{Hughes1}
Bruce Hughes.
\newblock Local similarities and the {H}aagerup property.
\newblock {\em Groups Geom. Dyn.}, 3(2):299--315, 2009.
\newblock With an appendix by Daniel S. Farley.

\bibitem{KK}
Marc Krasner and L{\'e}o Kaloujnine.
\newblock Produit complet des groupes de permutations et probl\`eme d'extension
  de groupes. {I}.
\newblock {\em Acta Sci. Math. Szeged}, 13:208--230, 1950.

\bibitem{schu}
Roger~C. Lyndon and Paul~E. Schupp.
\newblock {\em Combinatorial group theory}.
\newblock Classics in Mathematics. Springer-Verlag, Berlin, 2001.
\newblock Reprint of the 1977 edition.

\bibitem{Maslov}
A.~N. Maslov.
\newblock The cyclic shift of languages.
\newblock {\em Problemy Pereda\v ci Informacii}, 9(4):81--87, 1973.

\bibitem{Witza}
Stefan Witzel and Matt Zaremsky.
\newblock Thompson groups for systems of groups, and their finiteness
  properties.
\newblock {\em arXiv:1405.5491, 48 pages}, 2014.

\end{thebibliography}

\end{document}